\documentclass[reqno]{amsart}

\usepackage{xcolor}
\usepackage{amssymb}
\usepackage{amsfonts}
\usepackage{amsmath}
\numberwithin{equation}{section}

\newtheorem{theorem}{Theorem}[section]
\newtheorem{proposition}[theorem]{Proposition}
\newtheorem{lemma}[theorem]{Lemma}
\newtheorem{corollary}[theorem]{Corollary}

\theoremstyle{definition}
\newtheorem{definition}[theorem]{Definition}
\newtheorem{example}[theorem]{Example}

\newtheorem{remark}[theorem]{Remark}

\title[Pole-cancellation functions of higher order]{Analytic characterizations of Jordan chains by  pole cancellation functions of higher order}
 \author{Muhamed Borogovac  and Annemarie Luger}
 
 \address{Muhamed Borogovac\\
Actuarial Department\\
Boston Mutual Life Insurance Company\\
120 Royall Street \\
Canton, MA 0202, United States \\
Phone 781-770-0317}
\email{Muhamed$_-$Borogovac@bostonmutual.com}

\address{Annemarie Luger\\
Department of Mathematics\\ Stockholm University\\ SE-106 91 Stockholm \\Sweden}
\email{luger@math.su.se}
 
\begin{document}
\maketitle

\begin{abstract}
In this paper  the analytic characterization of generalized poles of operator valued generalized Nevanlinna functions (including the length of Jordan chains of the representing relation) is completed. In particular, given a Jordan chain { of the representing relation} of length $\ell$, we show that there exists a pole cancellation function of order at least $\ell$, and, moreover, the construction shows that it is of surprisingly simple form.  
\end{abstract}

\section{Introduction}

In this work    certain spectral properties of   operators (or relations) are characterized by analytic properties of corresponding matrix  or operator valued functions. More precisely, the focus lies on Jordan chains of self-adjoint operators in  Pontryagin spaces.  Our result gives a complete answer to a question by  M.G.\ Krein and H.\ Langer, posed in connection with generalized Nevanlinna functions. 

As a motivation we give a short description of two approaches that lead to this question.
Start with an easy example: Let $A$ be a Hermitian matrix, then the resolvent $(A-z)^{-1}$ is a matrix-valued function with poles precisely  at the eigenvalues of $A$. Obviously these poles are of order one, since the Hermitian matrix $A$ cannot have any Jordan chains.

In this paper we are interested in the more general situation, where Jordan chains can appear, as well as it is possible that an eigenvalue is not an isolated point of the spectrum. More precisely, let $\big(\mathcal K,[\,\cdot\,,\,\cdot\,]\big)$ be a Pontryagin space and $A$ a self-adjoint operator (or even relation) in this space\footnote{Note that this property can also be described by the following condition. Let $G$ be a self-adjoint operator in a Hilbert space such that {{$0\in\varrho(G)$ and}} $\sigma(G) \cap{\mathbb R ^-}$ consists of finitely many negative eigenvalues of finite multiplicity only. Then we consider operators or relations $A$ satisfying that $GA$ is self-adjoint in the Hilbert space.}.

We are then interested in the eigenvalues of $A$, in particular, in the structure of the algebraic eigenspaces; i.e.\,Jordan chains. To this end let $\mathcal H$ be a Hilbert space and $\Gamma_0:\mathcal H\to\mathcal K$ a bounded linear operator, denote its adjoint by $\Gamma_0^+$, and define the (matrix or operator valued) function $Q$ by 
\begin{equation}\label{resolvent}
Q(z):=\Gamma_0^+(A-z)^{-1}\Gamma_0,\qquad z\in\varrho({{A}}).
\end{equation}
If in applications $A$ is a differential operator then $\Gamma_0$ might  act on the boundary of the domain.

Instead of investigating $A$ itself, or its resolvent, we aim to describe the properties of Jordan chains (such as lengths and inner products between elements) by the function $Q$ instead. Note that the singularities of $Q$ belong to the spectrum of \nolinebreak $A$, however, in general the converse is not true. In order to assure equality  a  minimality assumption is needed,  see Proposition   \ref{2.2}.

In this   formulation of the problem the focus lies on the operator $A$, whereas the function \nolinebreak$Q$ appears as an auxiliary object. Conversely, one can also focus on functions that essentially arise from the above situation, namely \textit{generalized Nevanlinna functions}, $Q\in\mathcal N_\kappa(\mathcal H)$. These are $\mathcal L(\mathcal H)$-valued functions generalizing the class $\mathcal N_0(\mathbb C)$, which  consists of scalar functions mapping the upper half plane holomorphically into itself, see Definition \ref{def}. 

Recall that such functions admit a realization in a Pontryagin space, see Proposition \ref{2.2}. Essentially this means that for $\alpha\in\mathbb R$ the function $Q\in\mathcal N_\kappa(\mathcal H)$ can be written in the form 
\begin{equation}\label{repH}
 Q(z)=\Gamma_0(A-z)^{-1}\Gamma_0+H(z),
\end{equation}
with some $\Gamma_0\in\mathcal L(\mathcal H,\mathcal K)$ and $H$ is analytic at $\alpha$. 
Hence the question arises how the internal spectral structure of $A$  is reflected in analytic properties of $Q$.

When dealing with such questions the following difficulties appear: An eigenvalue need not be an isolated spectral point and hence it is possible that the corresponding algebraic eigenspace is degenerate. This amounts to the fact that no ``Laurent expansion'' is possible, with this we mean that given $\alpha$ it might be impossible to write $Q$ as 
$$
Q(z)=R(z)+Q_0(z),
$$ 
where $R$ is rational and $Q_0\in\mathcal N_0(\mathcal H)$ is a Nevanlinna function without negative squares, where Jordan chains cannot appear. As an example serves $Q(z)=\frac{\sqrt{z}}{z^2}$. Another major issue is the fact that both $Q$ and $\widehat Q(z):=-Q(z)^{-1}$ might have a generalized pole at the same point. The current paper reveals that the first difficulty, in fact, is only technical, whereas the second is intrinsic.

The problem of characterizing the eigenspace structure of $A$ in \eqref{repH} in terms of analytic properties of $Q$ has  been around from the time when generalized Nevanlinna functions have been introduced, around 40 years ago, cf.\,\cite{KL2}.

For scalar generalized Nevanlinna functions the non-positive part of the Jordan chain was characterized in \cite{L2} whereas as an analytic characterization of the whole chain (in terms of the asymptotic behaviour of $Q$) can be found in  \cite{HL}. The methods there are quite different from our approach. 

For matrix (and operator) generalized Nevanlinna functions the problem is much harder, since the singularities lie only in ``certain directions''. The appropriate tool here 
are so-called \textit{pole cancellation functions}. Essentially this is a vector valued function $\vec\eta$ vanishing at $\alpha$, the point under investigation, such that $Q(z)\vec\eta(z)$ is in some sense regular and does not vanish at the point $\alpha$, see Definitions \ref{def1} and \ref{def2}. 

It has partially been shown earlier, see Remark \ref{rem}, that  the existence of such a pole cancellation function (vanishing of order $\ell$) implies that $A$ has a Jordan chain (of length at least $\ell$).

 Conversely, the construction of a pole cancellation function -- given a Jordan chain -- has appeared to be much more demanding. The results available either make rather restrictive assumptions or cannot cover the whole Jordan chain, and, moreover, all these constructions are very technical and do not help so much in finding a pole cancellation function in a concrete situation, see Remark \ref{rem2} for more history of the problem.

In the present paper we not only show that for a Jordan chain of length $\ell$ there exists a pole cancellation function of order at least $\ell$, but also - not less important -  it is given in a surprisingly simple form, see Theorems \ref{main} and  \ref{final}. In a concrete situation this means that there is only a finite number of parameters that have to be found, when constructing the pole cancellation function.

The article is organized as follows. After this introduction and preliminaries the main results are in Section \ref{S3}. In \ref{ss}  a short collection of results on how the existence of a pole cancellation function implies the existence of a Jordan chain is given.  The core of this text is Section \ref{cc}, where the construction of the pole cancellation function  is done. We also illustrate the sharpness of this result (regarding order and the polynomial form) by  several examples, collected in \ref{examples}. The presentation is completed by a short conclusion.

\section{Preliminaries }

We start this section with the analytic definition of generalized Nevanlinna functions even though in the following mainly an alternative way of describing these functions, namely their realizations, will be used.
Let $\big(\mathcal H,(\,\cdot\,,\,\cdot\,)\big)$ be a Hilbert space and denote by $\mathcal L(\mathcal H)$ the set of bounded linear operators in $\mathcal H$, and let $\mathbb C^+:=\{z\in\mathbb C:{\rm Im} z>0\}$ denote the open upper half plane. 
\begin{definition}\label{def}
 An operator valued  function $Q:\mathcal D(Q)\subset\mathbb C\to\mathcal L(\mathcal H)$ is said to  belong to the \textit{generalized Nevanlinna class}, $\mathcal N_\kappa(\mathcal H)$, if it satisfies the following properties:\vspace{1mm}

\begin{itemize}
 \item $Q$ is meromorphic in $\mathbb C\setminus \mathbb R$,\vspace{1mm}
\item $Q(\overline z)^*=Q(z)$ for all $z\in\mathcal D(Q)$,\vspace{1mm}
\item the Nevanlinna kernel 
$$
N_Q(z,w):=\frac{Q(z)-Q(w)^*}{z-\overline w}\,\quad z,\,w\in\mathcal D(Q)\cap\mathbb C^+
$$
has $\kappa$ negative squares, i.e.\,for arbitrary $n\in\mathbb N,\, z_{1},  \ldots ,
z_{n}\in D(Q)\cap \mathbb C^+$ and $\vec h_{1},,\ldots \vec h_{n}\in \mathcal H$ the Hermitian matrix 
$\left( N_{Q}\left( z_{i},z_{j} 
\right)\vec h_{i},\vec h_{j} \right)_{i,j=1}^{n}$ has at most $\kappa $ 
negative eigenvalues, and  $\kappa$ is minimal with this property. 
\end{itemize}
\end{definition}
It is well known, see eg.\,\cite{KL2} and \cite{HSW},  that these functions can also be described by their realizations. 
\begin{proposition}\label{2.2}
A  function $Q$ with values in $\mathcal L(\mathcal H)$ is a generalized Nevanlinna function if and only if there exist a Pontryagin space $\big(\mathcal K,[\,\cdot\,,\,\cdot\,]\big)$, a self-adjoint relation $A$  in $\mathcal K$, a point $z_0\in\varrho(A)\cap\mathbb C^+$, and a bounded linear map $\Gamma:\mathcal H \to\mathcal K$ such that $Q$ can be written as
\begin{equation}\label{rep}
 Q(z)=Q(z_0)^*+(z-\overline{z_0})\Gamma^+\Big(I+(z-z_0)(A-z)^{-1}\Big)\Gamma\quad \text{for all } z\in\mathcal D(Q).
\end{equation}
Moreover, this realization can be chosen minimal, that is, 
$$
\mathcal K=\overline{span}\big\{\big(I+(z-z_0)(A-z)^{-1}\big)\Gamma\vec h,\,z\in\varrho(A),\,\vec h\in\mathcal H\big\}.
$$
 If the realization is minimal, then $Q\in\mathcal N_\kappa(\mathcal H)$ if and only if the negative index of the Pontryagin space equals $\kappa$. 
\end{proposition}

Recall that a linear relation can be seen as a multi-valued operator, see eg.\,\cite{DS}. 
With the abbreviation 
$$
\Gamma_z:=\big(I+(z-z_0)(A-z)^{-1}\big)\Gamma
$$
the following useful identities hold for $z,\,w\in\varrho(A)$:
\begin{equation}\label{0}
 \Gamma_w=\big(I+(w-z)(A-w)^{-1}\big)\Gamma_z
\end{equation}
\begin{equation}\label{1}
 \Gamma_w ^+\Gamma_z=\Gamma_{\overline z}^+\Gamma_{\overline w}=\frac{Q(z)-Q(\overline w)}{z-\overline w}
\end{equation}
\begin{equation}\label{2}
 (A-z_0)^{-1}\Gamma_z=\frac1{z-z_0}\big(\Gamma_z-\Gamma_{z_0}\big)
\end{equation}
\begin{remark}
 Note that  the minimal realization of a given function $Q$ is unique up to unitary equivalence and in this case $\varrho(A)={\rm hol}(Q)$.
\end{remark}
Hence in what follows, we refer to $A$ in a minimal realization of $Q$ as the \textit{representing relation} of $Q$ and we will be interested in its algebraic eigenspace structure.
\begin{definition}
A point  $\alpha\in\mathbb C\cup\{\infty\}$ is called \textit{generalized pole} of $Q\in\mathcal N_\kappa(\mathcal H)$ if $\alpha$ is an eigenvalue of the representing relation $A$ in a minimal realization of the form \nolinebreak \eqref{rep}.
\end{definition}

\section{Algebraic eigenspace and pole cancellation functions}\label{S3}

The main tool in the analytic characterization of generalized poles (zeros) are so-called pole cancellation functions (root functions), which  also have been used for ``ordinary'' poles (zeros) of  general matrix-valued meromorphic functions, cf.\, {\cite{GSi, GLR}}. 

Definitions and also notions vary in the literature. We give two versions, their relation will become clear by Theorem \ref{theorem-suff}. Here and in the following $z\hat\to\alpha$ denotes the non-tangential limit as $z$ tends to $\alpha\in\mathbb R$, and $\text{w-}\!\lim$ the weak limit, here in the definitions, in the Hilbert space $\mathcal H$.

\begin{definition}\label{def1}
Let $Q\in\mathcal N_\kappa(\mathcal H)$ and $\alpha\in\mathbb R$ be given. Denote by $\mathcal U_\alpha$ a neighbourhood of $\alpha$. A function $\vec\eta:\mathcal U_\alpha\cap\mathbb C^+\to\mathcal H$ is called \textit{pole cancellation function of $Q$ at \nolinebreak$\alpha$} if the following properties are satisfied:\vspace{1mm}
\begin{itemize}
\item[] {\bf (A) } $\text{w-}\!\lim\limits_{z\hat\to\alpha}\vec\eta(z)=\vec0$,\vspace{1mm}
\item[] {\bf (B) } $\text{w-}\!\lim\limits_{z\hat\to\alpha}Q(z)\vec\eta(z)=:\vec\eta_0$ exists and $\vec\eta_0\neq\vec0$,\vspace{1mm}
\item[] {\bf (C) } $\left(\dfrac{Q(z)-Q(z)^*}{z-\overline z}\vec\eta(z),\vec\eta(z)\right)$ is bounded as $z\hat\to\alpha$.\vspace{1mm}
\end{itemize}
We say that  $\vec\eta$ is a \textit{strong pole cancellation function} if, moreover, \vspace{2mm}
\begin{itemize}
\item[]  {\bf (C$_{s}$) } $\lim\limits_{z,w\hat\to\alpha}\left(\dfrac{Q(z)-Q(w)^*}{z-\overline w}\vec\eta(z),\vec\eta(w)\right)$ exists.
\end{itemize}
\end{definition}
As we will see, the existence of a pole cancellation function at a point $\alpha$ is sufficient for $\alpha$ to be a generalized pole of $Q$. However, in order to describe the whole algebraic eigenspace  of the representing relation (not only the eigenvectors) higher order derivatives will be  needed. 

\begin{definition}\label{def2}
A pole cancellation function $\vec\eta$ of $Q$ at $\alpha$ is of \textit{order}  $\ell\in\mathbb N$ if $\ell$ is the maximal number such that for all $0\leq j <\ell$ it holds
\begin{itemize}
\item[] {\bf (D) } $\text{w-}\!\lim\limits_{z\hat\to\alpha}\big(\vec\eta(z)\big)^{(j)}=\vec0$,
\item[] {\bf (E) }  $\text{w-}\!\lim\limits_{z\hat\to\alpha}\big(Q(z)\vec\eta(z)\big)^{(j)}=:\vec\eta_j$ exist and $\vec\eta_0\neq\vec0$,\vspace{1mm}
\item[] {\bf (F) } $\dfrac{d^{2j}}{dz^j\,d\overline z^j}\left(\dfrac{Q(z)-Q(z)^*}{z-\overline z}\vec\eta(z),\vec\eta(z)\right)$ is bounded as $z\hat\to\alpha$.\vspace{2mm}
\end{itemize}
A strong  pole cancellation function  $\vec\eta$ is said to be of \textit{order}  $\ell\in\mathbb N $, if   additionally also the following property is satisfied\vspace{2mm}
\begin{itemize}
\item[] {\bf (F$_{s}$) } $\lim\limits_{z,w\hat\to\alpha}\dfrac{d^{2j}}{dz^j\,d\overline w^j}\left(\dfrac{Q(z)-Q(w)^*}{z-\overline w}\vec\eta(z),\vec\eta(w)\right)$ exist for all $0\leq j <\ell$.
\end{itemize}
\end{definition}

Note that for $j=0$ conditions {\bf (D-F$_{s}$)}  become {\bf (A-C$_{s}$)}.  Also for historical reasons we have separated these two cases.

Let $Q\in\mathcal N_\kappa(\mathcal H)$ and $\alpha\in\mathbb R$ be given. The analytic characterization of the algebraic eigenspace of the representing relation with the help of pole cancellation functions is divided into two sections, corresponding to the two implications that have to be shown.

\subsection{Jordan chains by means of pole cancellation functions}\label{ss}

In this  section we assume that there exists a pole cancellation function of $Q$ at $\alpha$ of order \nolinebreak $\ell$ as introduced in Definitions \ref{def1} and \ref{def2}. Under these (or  stronger) assumptions partial results for the existence of a Jordan chain of the representing relation   have been obtained in \cite{BL,B,DLS93}, and \cite{Lu2}. The following theorem collects and completes these results, in particular, we  allow weaker conditions on the pole cancellation function. 

\begin{theorem}\label{theorem-suff}
Let $\alpha\in\mathbb R$ and $Q\in\mathcal N_\kappa(\mathcal H)$ be given with a minimal realization $\eqref{rep}$
$$
 Q(z)=Q(z_0)^*+(z-\overline{z_0})\Gamma^+\big(I+(z-z_0)(A-z)^{-1}\big)\Gamma,
$$
and let $\vec\eta:\mathcal U_\alpha\cap\mathbb C^+\to\mathcal H$ satisfy $\bf{(A)}$ and $\bf{(B)}$ in Definition $\ref{def1}$. Then the following hold: \vspace{1mm}
\begin{enumerate}
 \item If $\vec\eta$ satisfies even $\bf{(C)}$,  i.e.\,it  is a pole cancellation function, then $\alpha$ is a generalized pole of $Q$, more precisely, \label{one}
$$
\Gamma_z\vec\eta(z)\overset{weakly}{\longrightarrow}x_0,\quad \text{ as }z\hat\to\alpha, 
$$
where $x_0$ is an eigenvector of the representing relation $A$. 
\item If the pole cancellation function $\vec\eta$ satisfies even $\bf{(C_s)}$ then \label{two}
$$
\Gamma_z\vec\eta(z)\overset{strongly}{\longrightarrow}x_0,\quad \text{ as }z\hat\to\alpha. 
$$
\item\label{three}  If the pole cancellation function $\vec\eta$ satisfies also $\bf{(D)}$, $\bf{(E)}$, and $\bf{(F)}$ for some $\ell\in\mathbb N$ then for $0\leq j<\ell$
\begin{equation}\label{weak}
\frac1{j!}\Big(\Gamma_z\vec\eta(z)\Big)^{(j)}\overset{weakly}{\longrightarrow}x_j,\quad  \text{ as }z\hat\to\alpha, 
\end{equation}
and $x_0, x_1,\ldots,x_{\ell-1}$ form a Jordan chain of $A$ at $\alpha$. 
\item\label{four}
If $\vec\eta$ additionally satisfies $\bf{(F_s)}$ then the convergence in $\eqref{weak}$ is strong and for  $0\leq i,j<\ell$ it holds
\begin{equation}\label{orto}
\hspace*{1cm}\lim\limits_{z,w\hat\to\alpha}\frac1{i!j!}\frac{d^{i+j}}{dz^i\,d\overline w^j}\left(\frac{Q(z)-Q(w)^*}{z-\overline w}\vec\eta(z),\vec\eta(w)\right)=[x_i,x_j].
\end{equation}
\end{enumerate}
\end{theorem}

\begin{proof}
Recall, see \cite{BL}, that weak convergence can be characterized as follows:\vspace{1mm}
{A sequence $\big(f_{k}\big)_{k\in \mathbb N}$  in a Pontryagin space 
$\big(\mathcal K,[\,\cdot\,,\,\cdot\,]\big)$   converges weakly if and only if   $ \big([ f_{k},f_{k} ]\big)_{k\in\mathbb N}$\vspace{1mm} is bounded and  $\big([f_{k},u]\big)_{k\in \mathbb N}$  is a Cauchy sequence for all elements $u$ of some total subset $U$ of $\mathcal K$.  
}

We apply this to  $ \frac1{j!}\Big(\Gamma_z\vec\eta(z)\Big)^{(j)}$ as $z\hat\to\alpha$. By conditions \eqref{1} and  {\bf (F)}   we have that 
\begin{equation}\label{weaklim}
\left[\frac1{j!}\Big(\Gamma_z\vec\eta(z)\Big)^{(j)},\frac1{j!}\Big(\Gamma_z\vec\eta(z)\Big)^{(j)}\right]=\frac1{(j!)^2}\frac{d^{2j}}{dz^j\,d\overline z^j}\left(\frac{Q(z)-Q(z)^*}{z-\overline z}\vec\eta(z),\vec\eta(z)\right)
\end{equation}
is bounded as $z\hat\to\alpha$. By assumption the set $\{\Gamma_w\vec h,\,w\in\varrho(A),\vec h\in\mathcal H\}$ (and hence \vspace{1mm} also the possibly smaller set \vspace{1mm}$\{\Gamma_w\vec h,\,w\in\varrho(A)\setminus\{\alpha\},\vec h\in\mathcal H\}$) is a total set in $\mathcal K$. By {\bf (D)} and {\bf (E)} it follows that for all $w\in\varrho(A)\setminus\{\alpha\}$ and $\vec h\in\mathcal H$
$$
\left[\frac1{j!}\Big(\Gamma_z\vec\eta(z)\Big)^{(j)}, \Gamma_w\vec h\right]=\frac1{j!}\frac{d^{j}}{dz^j}\frac{(Q(z)\vec\eta(z),\vec h)-(\vec\eta(z),Q(w)\vec h)}{z-\overline w}
$$
converges as $z\hat\to\alpha$. Thus for $j<\ell$  the above mentioned characterization from \cite{BL} implies that  $\frac1{j!}\Big(\Gamma_z\vec\eta(z)\Big)^{(j)}$ converges weakly as $z\hat\to\alpha$,   and we denote the limit element by $x_j$. 

Next we show that $x_0,\ldots,x_{\ell-1}$ is a Jordan chain of $A$  at $\alpha$. With \eqref{2} we obtain for $0\leq j<\ell$ 
\begin{eqnarray*}
 (A-z_0)^{-1} x_j &=&  \text{w-}\!\lim_ {z\hat \to\alpha}(A-z_0)^{-1}\frac1{j!}\Big(\Gamma_z\vec\eta(z)\Big)^{(j)}\\
&=&  \text{w-}\!\lim_ {z\hat \to\alpha} \frac{d^{j}}{dz^j}\left(\frac1{j!(z-z_0)}\big(\Gamma_z\vec\eta(z)-\Gamma_{z_0}\vec\eta(z)\big)\right).
\end{eqnarray*}
Taking into account  {\bf (D)} this further equals 
\begin{eqnarray}
\hspace*{0.5cm}&=&  \text{w-}\lim_ {z\hat \to\alpha} \frac{d^{j}}{dz^j}\frac1{j!(z-z_0)} \Gamma_z\vec\eta(z) 
=\text{w-}\lim_ {z\hat \to\alpha}\sum_{k=0}^j\frac1{j!} \binom{j}{k} \frac{(-1)^{j-k}(j-k)!}{(z-z_0)^{j-k+1}}\Big(\Gamma_z\vec\eta(z)\Big)^{(k)}\nonumber\\
&=&\sum_{k=0}^j\frac{(-1)^{j-k}}{(\alpha-z_0)^{j-k+1}} x_k.\nonumber
\end{eqnarray}
For $j=0$ this is 
$$
(A-z_0)^{-1}x_0=\frac1{\alpha-z_0} x_0.
$$
Assume now $x_0=0$, then for all $\vec h\in\mathcal H$ we have
$$
0=[x_0,\Gamma\vec h]=\lim_{z\hat\to\alpha}[\Gamma_z\vec\eta(z),\Gamma\vec h]=\lim_{z\hat\to\alpha}\Big(\frac{Q(z)-Q(\overline{z_0})}{z-\overline{z_0}}\vec\eta(z),\vec h\Big)=\frac{(\vec\eta_0,\vec h)}{\alpha-\overline{z_0}},
$$
and hence  $\vec\eta_0=\vec0$. This contradicts {\bf (B)},   hence $x_0$ is an eigenvector of $A$ at $\alpha$, and \ref{one} is \nolinebreak shown. 

Rewriting 
$$
(A-z_0)^{-1} x_j=\sum_{k=0}^j\frac{(-1)^{j-k}}{(\alpha-z_0)^{j-k+1}} x_k , \qquad\text{ for }j=1,\ldots,\ell-1
$$
by using this relation again backwards for  $j-1$ yields
$$
(A-z_0)^{-1} x_j=\frac1{\alpha -z_0}\left(x_j-(A-z_0)^{-1}x_{j-1}\right),
$$
and hence $x_0,\ldots,x_{\ell-1}$ forms a Jordan chain of $A$ at $\alpha$, which proves \eqref{three}.

In order to show \eqref{two} and \eqref{four} a similar characterization for strong convergence is used, namely where the condition that $ \big([ f_{k},f_{k} ]\big)_{k\in\mathbb N}$ is bounded is substituted by the requirement that  $ \big([ f_{k},f_{n} ]\big)_{k,n\in\mathbb N}$ is a Cauchy-sequence. 
Then instead of \eqref{weaklim} we note that by \textbf{(F$_s$)} we have that 
$$
\left[\frac1{j!}\Big(\Gamma_z\vec\eta(z)\Big)^{(j)},\frac1{j!}\Big(\Gamma_w\vec\eta(w)\Big)^{(j)}\right]=\frac1{(j!)^2}\frac{d^{2j}}{dz^j\,d\overline w^j}\left(\frac{Q(z)-Q(w)^*}{z-\overline w}\vec\eta(z),\vec\eta(w)\right)
$$
converges as $z,w\hat\to\alpha$ and hence 
$$
\frac1{j!}\Big(\Gamma_z\vec\eta(z)\Big)^{(j)}\overset{strongly}{\longrightarrow}x_j,\quad  \text{ as }z\hat\to\alpha.
$$
From \eqref{three} we already know that $x_0,\ldots,x_{\ell-1}$ is a Jordan chain of $A$ and hence \eqref{two} and the first statement of \eqref{four} are shown.
Finally, as by \textbf{(F$_s$)} the limit in \eqref{orto} exists it follows that 
$$
\frac{d^{i+j}}{dz^i\,d\overline w^j}\left(\frac{Q(z)-Q(w)^*}{z-\overline w}\vec\eta(z),\vec\eta(w)\right)=
\left[\frac1{i!}\Big(\Gamma_z\vec\eta(z)\Big)^{(i)}, \frac1{j!}\Big(\Gamma_w\vec\eta(w)\Big)^{(j)}\right]$$
 for $0\leq i,j<\ell$ converges to $[x_i,x_j]$ as $z,w\hat\to\alpha$, 
which completes the proof.
\end{proof}

\begin{remark}\label{rem}
Statement \eqref{one} was basically shown in \cite{BL}, only  property \textbf{(C)} was replaced by a stronger property (but weaker than \textbf{(C$_s$)}), namely that the limit
$$\lim\limits_{z\hat\to\alpha}\left(\frac{Q(z)-Q(z)^*}{z-\overline z}\vec\eta(z),\vec\eta(z)\right)\quad \text{  }
$$
exists. Already there (and explicitly mentioned in \cite{DLS93}) only the weaker condition corresponding to \textbf{(C)} is actually  used in the proof. However,   the above limit  plays an important role there, as it is shown that 
$$
\lim\limits_{z\hat\to\alpha}\left(\frac{Q(z)-Q(z)^*}{z-\overline z}\vec\eta(z),\vec\eta(z)\right)\geq[x_0,x_0]
$$
and hence the non-positivity of this limit implies that $x_0$ is a non-positive element, which was an important issue  in these papers.

Statement \eqref{two} can be found in    \cite{BL}   in the proof of Proposition 1.
 Statements \eqref{three} and \eqref{four} are  new in this generality.  Originally such results have been proven in \cite{B} and \cite{Lu2}. In \cite{B}  holomorphy conditions were used, as only meromorphic $Q$ are treated there. In the general case the proof was given in \cite{Lu2}, however, using {\bf (F$_s$)} and  an (unnecessary) strong variant of {\bf (E)}. Note that with  {\bf (F)} instead of {\bf (F$_s$)} only inequalities could have been achieved in \ref{orto}.

\end{remark}\enlargethispage{\baselineskip}\enlargethispage{\baselineskip}

\subsection{Pole   cancellation functions by means of Jordan chains}\label{cc}

The main result of this paper is Theorem \ref{main} together with Theorem \ref{final}. 
Recall that a function $Q\in\mathcal N_\kappa(\mathcal H)$ is called \textit{regular} if there exists a point $\gamma\in\mathbb C^+$ such that $Q(\gamma)$ is boundedly invertible. For a regular function $Q$ a point  $\alpha\in\mathbb C\cup\{\infty\}$ is called \textit{generalized zero} of $Q$ if it is a  generalized pole of $\widehat Q(z):=-Q(z)^{-1}$.

\begin{theorem}\label{main}
Let the regular generalized Nevanlinna function $Q$  be given with a minimal realization \eqref{rep}
$$
 Q(z)=Q(z_0)^*+(z-\overline{z_0})\Gamma^+\big(I+(z-z_0)(A-z)^{-1}\big)\Gamma
$$
and assume that $\alpha\in\mathbb R$ is not a generalized zero of $Q$.

If $\alpha\in\mathbb R$ is a generalized pole of $Q$, that is $\alpha\in\sigma_p(A)$, and $x_0,x_1,\ldots,x_{\ell-1}$ is  a  Jordan chain of $A$  at $\alpha$, then
\begin{equation}\label{eta}
\vec\eta \left( z \right):=(z-\overline{z_0}){Q\left( z \right)}^{-1}\Gamma^{+}\Big( 
x_{0}+\left( z-\alpha \right)x_{1}+\ldots +{(z-\alpha )}^{\ell-1}x_{\ell-1} 
\Big) .
\end{equation}
is a strong pole cancellation function of $Q$ at $\alpha$ of order at least $\ell$.

\end{theorem}

The factor $(z-\overline{z_0})$ is not necessary for the above statement, however,  as detailed in the following corollary, it allows us to recover the original Jordan chain from the pole cancellation function as in Theorem \ref{theorem-suff}.

\begin{corollary}\label{back}
With the notation from Theorem $\ref{main}$  for $0\leq i,j<\ell$ it holds
$$
\frac1{j!}\left(\Gamma_z\vec\eta(z)\right)^{(j)}\overset{strongly}{\longrightarrow}x_j,\quad \text{ as }z\hat\to\alpha, 
$$
and 
$$\lim\limits_{z,w\hat\to\alpha}\frac1{i!j!}\frac{d^{i+j}}{dz^i\,d\overline w^j}\left(\frac{Q(z)-Q(w)^*}{z-\overline w}\vec\eta(z),\vec\eta(w)\right)=[x_i,x_j] .$$
\end{corollary}

\begin{remark}\label{simple}
If $Q$ can be written in the simpler form \eqref{resolvent} then $\vec\eta$ in Theorem \ref{main} can be chosen as $\vec\eta(z):=Q(z)^{-1}\Gamma_0^+\Big( 
x_{0}+\left( z-\alpha \right)x_{1}+\ldots +{(z-\alpha )}^{\ell-1}x_{\ell-1} 
\Big)$ .
\end{remark}

The proof of Theorem \ref{main} is based on two technical lemmas. Recall first (see eg.\,\cite{Lu1,LaLu}) that if $Q\in\mathcal N_\kappa(\mathcal H)$ has the realization 
$$
 Q(z)=Q({z_0})^*+(z-\overline{z_0})\Gamma^+\big(I+(z-z_0)(A-z)^{-1}\big)\Gamma
$$
then the inverse function $\widehat Q$ has the realization
$$
 \widehat Q(z)=\widehat Q({z_0})^*+(z-\overline{z_0})\widehat\Gamma^+\big(I+(z-z_0)(\widehat A-z)^{-1}\big)\widehat\Gamma,
$$
where 
\begin{equation}\label{gammahat}
\widehat\Gamma:=\Gamma\widehat Q(z_0)=-\Gamma Q(z_0)^{-1}
\end{equation}
and 
\begin{equation}\label{Ahat}
 (\widehat A-z)^{-1}=(A-z)^{-1}+\Gamma_z\widehat Q(z)\Gamma_{\overline z}^+.
\end{equation}

\begin{lemma}\label{lemma1}
 With the above notations and the definition
$$
\widehat \Gamma_z:=\big(I+(z-z_0)(\widehat A-z)^{-1}\big)\widehat\Gamma
$$
it holds
\begin{equation}\label{gammahat1}
\widehat\Gamma_z=\Gamma_z\widehat Q(z)
\end{equation}
and
\begin{equation}\label{Ahat1}
 ( A-z)^{-1}=(\widehat A-z)^{-1}+\widehat\Gamma_z Q(z)\widehat\Gamma_{\overline z}^+.
\end{equation}
\end{lemma}
\begin{proof}
We start from the definition of $\widehat\Gamma_z$ and use \eqref{Ahat} and \eqref{1}.
\begin{eqnarray*}
 \widehat\Gamma_z &=& -\left(I+(z-z_0)(\widehat A-z)^{-1}\right)\Gamma Q(z_0)^{-1}\\
&=& -\Big(I+(z-z_0)(A-z)^{-1}-(z-z_0)\Gamma_z Q(z)^{-1}\Gamma_{\overline z}^+\Big)\Gamma Q(z_0)^{-1}\\
&=& -\Gamma_zQ^{-1}(z_0)+\Gamma_zQ(z)^{-1} \big(Q(z)-Q(z_0)\big)Q(z_0)^{-1}= \Gamma_z\widehat Q(z).
\end{eqnarray*}
This implies
$$
\widehat\Gamma_z Q(z)\widehat\Gamma_{\overline z}^+=\Gamma_z\widehat Q(z)Q(z)\widehat Q(z)\Gamma_{\overline z}^+=-\Gamma_z\widehat Q(z)\Gamma_{\overline z}^+,
$$
and hence also \eqref{Ahat1} is shown.
\end{proof}

In what follows, given a Jordan chain $x_0,\ldots,x_{\ell-1}$ of $A$ at $\alpha$ we use the following abbreviation
$$
x(z):= x_0+(z-\alpha)x_1+\ldots +(z-\alpha)^{\ell-1}x_{\ell-1}.
$$
Note that then 
\begin{equation}\label{xz}
(A-\overline{z_0})^{-1}x(z)=\frac1{z-\overline{z_0}}\Big(x(z)+(z-\alpha)^\ell(A-\overline{ z_0})^{-1}x_{\ell-1}\Big).
\end{equation}
\begin{lemma}\label{rewrite}
 With the above notations it holds
\begin{equation}\label{only}
\widehat\Gamma_z\Gamma^+x(z)=\frac{-1}{z-\overline{z_0}}\Big(x(z)+(z-\alpha)^\ell\big(I+(z-\overline{z_0})(\widehat A-z)^{-1}\big) (A-\overline{z_0})^{-1}x_{\ell-1}\Big).
\end{equation}

\end{lemma}
\begin{proof}
The definition of $\widehat\Gamma_z$ gives 
$$
-\widehat\Gamma_z\Gamma^+x(z)=\big(I+(z-z_0)(\widehat A-z)^{-1}\big) \Gamma Q(z_0)^{-1}\Gamma^+ x(z).
$$ 
In order to rewrite $\Gamma Q(z_0)^{-1}\Gamma^+$ we  use \eqref{gammahat1} and an identity as \eqref{0} for the relation $\widehat A$
\begin{eqnarray*}
\Gamma Q(z_0)^{-1}\Gamma^+ &=& \widehat \Gamma Q(z_0)Q(z_0)^{-1}Q(\overline{z_0})\widehat\Gamma^+=\widehat \Gamma_{z_0} Q(\overline{z_0})\widehat\Gamma_{z_0}^+\\
&=&\big(I+(z_0-\overline{z_0})(\widehat A-z_0)^{-1}\big)\widehat \Gamma_{\overline z_0} Q(\overline{z_0})\widehat\Gamma_{z_0}^+.
\end{eqnarray*}
Applying the resolvent identity and  using  relation  \eqref{Ahat1}  we  obtain
$$
-\widehat\Gamma_z\Gamma^+x(z)={\big(I+(z-\overline{z_0})(\widehat A-z)^{-1}\big)}\Big((A-\overline{z_0})^{-1}-(\widehat A-\overline{z_0})^{-1}\Big)x(z)
$$
With \eqref{xz} and again the resolvent identity for $\widehat A$ this further equals
$$
\frac1{z-\overline{z_0}}\Big(I+(z-\overline{z_0})(\widehat A-z)^{-1}\Big)\Big(x(z)+(z-\alpha)^\ell(A-\overline{z_0})^{-1}x_{\ell-1}\Big)-(\widehat A-z)^{-1}x(z)
$$
\begin{eqnarray*}
 &=&
\frac1{z-\overline{z_0}}\Big(x(z)+(z-\alpha)^\ell\big(I+(z-\overline{z_0})(\widehat A-z)^{-1}\big) (A-\overline{z_0})^{-1}x_{\ell-1}\Big)
\end{eqnarray*}
and hence the lemma is proved.
\end{proof}

\begin{proof}(Theorem \ref{main})
 We start by rewriting $\vec\eta(z)$ from \eqref{eta}, where, in particular,  \eqref{only} is used   in the fourth equality
\begin{eqnarray}
\vec\eta(z) &=& (z-\overline{z_0})Q(z)^{-1}\Gamma^+x(z)\nonumber\\
&=& (z-\overline{z_0})\Big(Q(\overline z_0)^{-1}-(z-\overline{z_0})\widehat\Gamma^+\widehat\Gamma_z\Big)\Gamma^+ x(z)\nonumber\\
&=& (z-\overline{z_0})\widehat\Gamma^+\Big(-x(z)-(z-\overline{z_0})\widehat\Gamma_z\Gamma^+x(z)\Big)\nonumber\\
&=& (z-\overline{z_0})(z-\alpha)^\ell\widehat\Gamma^+\Big(I+(z-\overline{z_0})(\widehat A-z)^{-1}\Big)(A-\overline{z_0})^{-1}x_{\ell-1}\label{eta2}
\end{eqnarray}
The crucial observation is that due to the assumption that $\alpha$ is not a generalized zero of $Q$, this is  $\alpha\not\in\sigma_p(\widehat A)$,  it holds as $z \hat\to \alpha$
$$
 (z-\alpha)(\widehat A-z)^{-1}\overset{strongly}{\longrightarrow}0
$$
and, moreover, for $ j<\ell$
\begin{equation*}\label{der}
\frac{d^j\,\,}{dz^j}\left((z-\alpha)^\ell\big(\widehat A-z\big)^{-1}\right)\overset{strongly}{\longrightarrow} 0.
\end{equation*}
In what follows we show that $\vec\eta$ satisfies the conditions from Definitions \ref{def1} and \ref{def2}.
\begin{itemize}
\item[] {\bf (A, D)} Formula \eqref{eta2}  and the above observations imply
 $\vec\eta^{\,(j)}\overset{strongly}{\longrightarrow} \vec0$ for $  j<\ell$.
\item[] {\bf (B, E\,)}
The definition of $\vec\eta$ gives
$$
\hspace*{1cm} \Big(Q(z)\vec\eta(z)\Big)^{(j)}= \Big((z-\overline{z_0})\Gamma^+\big({\textstyle x_0+(z-\alpha)x_1+\ldots +(z-\alpha)^{\ell-1}x_{\ell-1}}\big)\Big)^{(j)}
$$
and hence the limits in {\bf (E)} exist even strongly, in particular,
$$
Q(z)\vec\eta(z)=(z-\overline z_0)\Gamma^+x(z)\to(\alpha-\overline{z_0})\Gamma^+x_0, \quad \text{ as }z \hat\to \alpha.
$$
Let us assume $\Gamma^+x_0=\vec 0$. As $x_0$ is an eigenvector of $A$ at $\alpha$ this implies 
$$
\big[\big(I+(z-z_0)(A-z)^{-1}\big)\Gamma\vec h, x_0\big]=\frac{\alpha-z_0}{\alpha-z}(\vec h,\Gamma^+x_0)=0
$$
for all $\vec h\in\mathcal H$ and $z\in\varrho(A)$ and minimality would hence imply $x_0=0$. This contradicts the fact that $x_0$ is an eigenvector and thus $\Gamma^+x_0\neq\vec0$.\vspace{2mm}

\item[] {\bf (C$_s$, F$_s$)} In order to show the existence of the limits we note that 
$$
\left(\frac{Q(z)-Q(\overline w)}{z-\overline w}\vec\eta(z),\vec\eta(w)\right)=[\Gamma_z\vec\eta(z),\Gamma_w\vec\eta(w)]
$$
and rewrite $\Gamma_z\vec\eta(z)$ with the help of \eqref{gammahat1} and  \eqref{only}
\begin{eqnarray*}
\Gamma_z\vec\eta(z) &=& -(z-\overline{z_0})\widehat\Gamma_z Q(z)Q(z)^{-1}\Gamma^+x(z)= -(z-\overline{z_0})\widehat\Gamma_z \Gamma^+x(z)\\[1mm]
&=& x(z)+(z-\alpha)^\ell\big(I+(z-\overline{z_0}) (\widehat A-z)^{-1}\big)(A-\overline{z_0})^{-1} x_{\ell-1}
\end{eqnarray*}
Hence $\Gamma_z\vec\eta(z)$  can be written as 
\begin{equation}\label{gammah}
 \Gamma_z\vec\eta(z)=x(z)+h_\ell(z),
\end{equation}
where  for  $0\leq j<\ell$
$$
\frac{d^j\,}{dz^j} h_\ell(z)\overset{strongly}{\longrightarrow}0 \quad \text{ as } z\hat\to\alpha.
$$
This implies
$$
[\Gamma_z\vec\eta(z),\Gamma_w\vec\eta(w)]=[x(z)+h_\ell(z),x(w)+h_\ell(w)]
$$
and hence the limits in \textbf{(C$_s$), (F$_s$)} exist.
\end{itemize}
This finishes the proof.
\end{proof}

\begin{proof}(Corollary \ref{back})
 Formula \eqref{gammah} implies   for $0\leq j<\ell$
$$
\frac1{j!}\Big(\Gamma_z\vec\eta(z)\Big)^{(j)}=\frac1{j!}\big(x(z)+h_\ell(z)\big)^ {(j)}\to x_j\quad \text{ as } z\hat\to\alpha{{,}}
$$
and 
\begin{eqnarray*}
 && \lim\limits_{z,w\hat\to\alpha}\frac1{i!j!}\frac{d^{i+j}}{dz^i\,d\overline w^j}\left(\frac{Q(z)-Q(w)^*}{z-\overline w}\vec\eta(z),\vec\eta(w)\right)\\
&=& \lim\limits_{z,w\hat\to\alpha}\frac1{i!j!}\frac{d^{i+j}}{dz^i\,d\overline w^j}[\Gamma_z\vec\eta(z),\Gamma_w\vec\eta(w)]\\
&=& \lim\limits_{z,w\hat\to\alpha}\frac1{i!j!}\frac{d^{i+j}}{dz^i\,d\overline w^j}
[x(z)+h_\ell(z),x(w)+h_\ell(w)]\\
&=&  \lim\limits_{z,w\hat\to\alpha}\frac1{i!j!}\frac{d^{i+j}}{dz^i\,d\overline w^j}
[x(z),x(w)]=[x_i,x_j],
\end{eqnarray*}
which finishes the proof.
\end{proof}

\begin{remark}
Note that in the last statement of the corollary also   pole cancellation functions corresponding to different Jordan chains can be employed. This makes it possible to describe all inner products in the algebraic eigenspace. 
\end{remark}

In the main theorem it was shown that given a Jordan chain $x_0,x_1,\ldots,x_{\ell-1}$ the pole cancellation function $\vec\eta(z)$ is of order at least $\ell$. So the question arises how the order of $\vec\eta$ is related to the maximality of the Jordan chain. We summarize our observations in the following corollary.
\begin{corollary}\label{cororder}
Let the pole cancellation function  $\vec\eta$ be given as in Theorem $\ref{main}$.
\begin{enumerate}
 \item If $x_0,\ldots,x_{\ell-1}$ is a maximal Jordan chain of $A$ then $\vec\eta$ has order $\ell$.
\item Conversely, if $\vec\eta$ has order $\ell$, then $x_0,\ldots,x_{\ell-1}$ need not be maximal.
\item If $x_0,\ldots,x_{\ell-1}$ is not maximal then  the order of $\eta$ is $\ell$ or larger than $\ell$ and there are examples for both situations.
\end{enumerate}
\end{corollary}
\begin{proof}
 (1) follows from Corollary \ref{back}. Example \ref{ex1}(a)  shows (2) and examples for (3) are given in Example \ref{ex1}(a) and (b), respectively.
\end{proof}

In Theorem \ref{main} there are still   restrictive assumptions, namely that $Q$ is regular and  $\alpha$ is not a generalized zero. These restrictions are removed in the following theorem. 

\begin{theorem}\label{final}
Let $Q\in\mathcal N_\kappa(\mathcal H)$ be given  and assume that   $\alpha\in\mathbb R$ is a generalized pole of $Q$ such that the representing relation $A$ has a Jordan chain at $\alpha$ of length $\ell$.

Then there exists a strong pole cancellation function $\vec\eta(z)$ of order at least $\ell$ of the form
$$
\vec\eta(z)=\big(Q(z)+S\big)^{-1}\vec p(z),
$$
where $S=S^*\in\mathcal L(\mathcal H)$ and $\vec p(z)$ is an $\mathcal H$-valued polynomial of degree $<\ell$.
\end{theorem}
\begin{proof}
It follows from the proof of Theorem 3.7 in \cite{Lu2} that for every $Q\in\mathcal N_\kappa(\mathcal H)$ and $\alpha\in\mathbb R$ there exists a self-adjoint $S\in\mathcal L(\mathcal H)$ such that the function $Q(z) +S$ is regular and  does not have $\alpha$ as   a generalized zero. It is easy to check that a pole cancellation function of $Q$ at $\alpha$ is a pole cancellation function of $Q+S$ and conversely (and all the limits in the definitions coincide). Then we choose $\vec\eta$ as the pole cancellation function constructed in Theorem \ref{main} for $Q+S$.
\end{proof}

\begin{remark}
{  By definition the point $\infty$ is a generalized pole of $Q$ if and only if $0$ is a generalized pole of the function $\widetilde Q(z):=Q(-\frac1z)$ (with the same multiplicities).  

If we then want to express  Definition \ref{def1} in terms of $Q$ rather than $\widetilde Q $ then  {\bf (A)} and {\bf (B)} remain unchanged, whereas the quotient in {\bf (C$_s$)} (and {\bf (C)} accordingly) has to be substituted by 
$$
z\overline w\left(\dfrac{Q(z)-Q(w)^*}{z-\overline w}\vec\eta(z),\vec\eta(w)\right).
$$ 
However, in Definition \ref{def2} all terms including derivatives become more involved, since inner derivatives from the function $-\frac1z$ come into play. 
}
\end{remark}

We want to point out that the assumption in Theorem \ref{main} is not only technical. 
Indeed, if \nolinebreak $\alpha$ is a generalized zero of $Q$ then  $\vec\eta$ from Theorem \ref{main} need not be a pole cancellation function or if it is, its order can be less than $\ell$ from the construction, see Example \ref{ex3}.
Moreover, in this case it might  happen that it is not possible to choose a pole cancellation function of the form $\vec\eta(z)=Q(z)^{-1}\vec p(z)$, with a suitable polynomial  \nolinebreak$p$. See Example \ref{ex2} for an illustration of this fact.
We finish this section with a historic remark.

\begin{remark}\label{rem2}
  Several results  have originally been proven for generalized zeros (not poles), but they immediately can be translated for generalized poles, as generalized zeros of $Q$ by definition are generalized poles of the inverse function $\widehat Q$. 

The first result of an analytic characterization  for non-scalar functions can be found in  \cite{BL}. There a generalized zero is characterized by the existence of a function $\phi$ with properties that in our notation correspond to the fact that $Q(z)^{-1}\phi(z)$ is a pole cancellation function of $Q$, but it is even mentioned that in the general situation they cannot characterize the multiplicity of the generalized pole. In other papers assumptions are made that guarantee that the algebraic eigenspace is not degenerate and hence not ortho-complemented, namely \cite{B} deals with meromorphic functions, not necessarily generalized Nevanlinna functions, whereas in \cite{DLS93} embedded eigenvalues are considered, but the further assumptions even rule out the existence of Jordan chains of length $>1$. 

In \cite{Lu2} a different approach (via the factorization of $Q$) enables to avoid such assumptions, but the drawback there is that a pole cancellation function can be constructed only for the non-positive part of a Jordan chain (this might be only half the chain)  and, moreover, this construction is quite complicated and non-constructive. 

\end{remark}

{ \subsection{Generalized zeros} \label{zeros}
By applying the above results to the inverse function corresponding characterizations for generalized zeros can be obtained as well. For the convenience of the reader we give some more details here. 

\begin{definition}
A point $\beta\in\mathbb C\cup\{\infty\}$ is called a {\it generalized zero}  of $Q\in\mathcal N_\kappa(\mathcal H)$ if it is a generalized pole of the function $\widehat Q(z)=-Q(z)^{-1}$. 
\end{definition}
Instead of pole-cancellation functions one works here with so-called root functions.

\begin{definition}\label{def11}
Let $Q\in\mathcal N_\kappa(\mathcal H)$ and $\beta\in\mathbb R$ be given. Denote by $\mathcal U_\beta$ a neighbourhood of $\beta$. A function $\vec\xi:\mathcal U_\beta\cap\mathbb C^+\to\mathcal H$ is called \textit{root function of $Q$ at \nolinebreak$\beta$} of \textit{order}  $\ell\in\mathbb N$ if $\ell$ is the maximal number such that for all $0\leq j <\ell$ it holds:\vspace{1mm}
\begin{itemize}
\item[] {\bf (K) } $\text{w-}\!\lim\limits_{z\hat\to\beta}\big(\vec\xi(z)\big)^{(j)}=:\vec\xi_j$ exist and $\vec\xi_0\neq\vec0$,
\item[] {\bf (L) }  $\text{w-}\!\lim\limits_{z\hat\to\beta}\big(Q(z)\vec\xi(z)\big)^{(j)}=\vec0$,\vspace{1mm}
\item[] {\bf (M) } $\dfrac{d^{2j}}{dz^j\,d\overline z^j}\left(\dfrac{Q(z)-Q(z)^*}{z-\overline z}\vec\xi(z),\vec\xi(z)\right)$ is bounded as $z\hat\to\beta$.\vspace{2mm}
\end{itemize}
It is called {\it strong}  pole cancellation function    of  {order}  $\ell\in\mathbb N $, if   additionally also the following property is satisfied\vspace{2mm}
\begin{itemize}
\item[] {\bf (M$_{s}$) } $\lim\limits_{z,w\hat\to\beta}\dfrac{d^{2j}}{dz^j\,d\overline w^j}\left(\dfrac{Q(z)-Q(w)^*}{z-\overline w}\vec\xi(z),\vec\xi(w)\right)$ exist for all $0\leq j <\ell$.
\end{itemize}
\end{definition}

Note that a vector function $\vec\xi(x)$ is a (strong) root function of $Q$ at $\beta$ if and only if $\vec\eta(z):=Q(z)\vec\xi(z)$ is a (strong) pole cancellation function of $\widehat Q$ at $\beta$ and, moreover, the respective orders coincide. Hence all the above statements can be rewritten in terms of generalized zeros and root functions. We restrict the presentation to the following theorem.

\begin{theorem}
Let  $Q\in\mathcal N_\kappa(\mathcal H)$  and $\beta\in\mathbb R$ be given. Then the root functions of $Q$  of order at least $\ell$ at $\beta$  are in one-to-one correspondence to the Jordan chains of length $\ell$ of the representing relation of the function $\widehat Q$. In particular,  the root function can always be chosen in the form $\vec\xi(z)=(I+SQ(z))^{-1}\vec p(z)$, where $S=S^*$ is a constant and $\vec p(z)$ an $\mathcal H$-valued polynomial. 
\end{theorem}

Note, however, that it can happen that it is not possible to choose $\vec \xi$ as a polynomial, that is $S=0$, see Example \ref{ex2}. 
}
\subsection{Examples} \label{examples}

Here we collect the examples that were referred to in the text. For simplicity we use the simpler form of $\vec\eta$ mentioned in Remark \ref{simple}. Note that all realizations given here are chosen minimal.
\begin{example}\label{ex1} 
Let $\mathcal K=\mathbb C^2$ and define with the Gram matrix 
$G=\begin{pmatrix} 0 & 1 \\ 1 & 0 \end{pmatrix}$ the  inner product in $\mathcal K $. Then the operator $A=\begin{pmatrix} 0 & 1 \\ 0 & 0 \end{pmatrix}$ is self-adjoint in $\mathcal K$.
\begin{enumerate}
 \item[(a)] With $\Gamma_{1}:=\begin{pmatrix} 1 & 0 \\ 0 & 1 \end{pmatrix}$ define the $\mathcal N_1(\mathbb C^2)$-function
$$
Q_1(z):=\Gamma_{1}^+(A-z)^{-1}\Gamma_{1}=\begin{pmatrix} 0 & -\dfrac1z \\[3mm] -\dfrac1z & -\dfrac1{z^2} \end{pmatrix}.
$$
Choosing the non-maximal Jordan chain consisting of the vector $x_0=\begin{pmatrix} 1 \\ 0  \end{pmatrix}$ only we obtain
$$
\vec\eta_1(z):=Q_1(z)^{-1}\Gamma_{1}^+x_0=\begin{pmatrix} -z \\ 0 \end{pmatrix} \quad \text{ and }\quad Q_1(z)\vec\eta_1(z)=\begin{pmatrix} 0 \\ 1 \end{pmatrix}.
$$
Hence  $\vec\eta_1$ is of order $1$.\vspace{3mm}
\item[(b)] With  $\Gamma_{2}:=\begin{pmatrix} -1 \\ 1 \end{pmatrix}$ define the $\mathcal N_1(\mathbb C)$-function
$$
Q_2(z):=\Gamma_{2}^+(A-z)^{-1}\Gamma_{2}=\frac{2z-1}{z^2}.
$$
 Choosing  again the Jordan chain  $x_0$   we now obtain
$$
\hspace*{1.3cm}\vec\eta_2(z):=Q_2(z)^{-1}\Gamma_{2}^+x_0=\frac{z^2}{2z-1} \text{ and }   Q_2(z)\vec\eta_2(z)=1.
$$
So we found  that the order of $\vec\eta_2$ equals  $2$, even if this pole cancellation function was constructed by a Jordan chain of (non-maximal) length $1$. This can be explained\vspace{2mm} by the fact that also for the Jordan chain $x_0,\,x_1$ with $x_1=\begin{pmatrix} 1 \\ 1 \end{pmatrix}\in\ker\Gamma_{2}^+$ it holds
$$
\hspace*{10mm}\vec\eta_2(z)=Q_2(z)^{-1}\Gamma_{2}^+\big(x_0+(z-\alpha)x_1\big)=Q_2(z)^{-1}\Gamma_{2}^+x_0.
$$
\end{enumerate}
\end{example}

\begin{example}\label{ex2}
 Consider the function \vspace{2mm} 
$$
Q_3(z):=\begin{pmatrix} z^2\sqrt{z} & 1  \\[2mm] 1 & -\dfrac1{z^2\sqrt z} \end{pmatrix}\in\mathcal N_2(\mathbb C^2),
$$
for  which $\alpha=0$ is both a generalized pole and zero since both $Q_3$ and  
$$
\widehat Q_3(z)=\dfrac12\begin{pmatrix}   -\dfrac1{z^2\sqrt z} & -1  \\[2mm] -1 & z^2\sqrt{z}\end{pmatrix}
$$
have a generalized pole at $\alpha=0$.  
 Let us now assume that there exists\vspace{2mm}  a pole cancellation function of $Q_3$ at $\alpha=0$ of the form $\vec \eta(z)=Q_3(z)^{-1}\begin{pmatrix} p_1(z) \\ p_2(z) \end{pmatrix}$ where $p_1$ and $p_2$ are polynomials. They can be written as
\begin{eqnarray*}
 p_1(z)&=&a_0+a_1z+a_2z^2+\mathcal O(z^3)\\
p_2(z) &=& b_0+\mathcal O(z), 
\end{eqnarray*}
as $z\hat\to0$. This gives 
$$
\vec \eta (z)=\frac12\begin{pmatrix} \dfrac{a_0}{z^2\sqrt z}+ \dfrac{a_1}{z\sqrt z}+\dfrac{a_2}{\sqrt z}+\mathcal O(\sqrt z)\\[3mm] a_0+\mathcal O(z) \end{pmatrix}\quad\text{ as }z\hat\to0
$$
and property {\bf (A)} implies $a_0=a_1=a_2=b_0=0$. Hence $Q_3(z)\vec \eta(z)=\vec p(z)\to\vec0$ as $z\hat\to0$ and $\vec \eta$ cannot be a pole cancellation function for $Q_3$. However, it is easy to check that $\alpha=0$ is a generalized pole and 
$$
Q_3(z)^{-1} \begin{pmatrix} -z^2\sqrt z    \\ 1  \end{pmatrix}
$$
is a pole cancellation function for $Q_3$ at $\alpha=0$.

{ A similar calculation shows that the root function $\vec\xi$ of $Q_3$ at $\beta=0$ cannot be a polynomial. }
\end{example}

\begin{example}\label{ex3}
The function 
$$
Q_4(z)=\left(
\begin{array}{cc}
 \dfrac{1}{(z-1) z^2} & \dfrac{1}{z} \\[4mm]
 \dfrac{1}{z} & \dfrac{-2 z-1}{(z+1)^3}
\end{array}
\right) 
$$
is a generalized Nevanlinna function for which $\alpha=0$ is both a pole and zero as also  
$$
\widehat Q_4(z)=\left(
\begin{array}{cc}
 \dfrac{(z-1) (2 z+1)}{z (z+2)} & \dfrac{(z-1) (z+1)^3}{z^2 (z+2)} \\[4mm]
 \dfrac{(z-1) (z+1)^3}{z^2 (z+2)} & -\dfrac{(z+1)^3}{z^3 (z+2)} \\
\end{array}
\right)
$$ has a pole at $0$.

We are going to show that $\vec\eta$ defined as in \eqref{eta} not necessarily is a pole cancellation function. To this end we need a minimal representation of $Q_4$. 
Let the space $\mathcal K_4=\mathbb C^6$,  and let the  Gram matrix $G_4$ and the operator $A_4$  be given as follows
$$G_4=\textstyle\begin{pmatrix} 0 & 1 & 0 & 0 & 0 & 0 \\
 1 & 0 & 0 & 0 & 0 & 0 \\
 0 & 0 & -1 & 0 & 0 & 0 \\
 0 & 0 & 0 & 0 & 0 & -1 \\
 0 & 0 & 0 & 0 & -1 & 0 \\
 0 & 0 & 0 & -1 & 0 & 0 \end{pmatrix}\qquad \text{ and } \qquad  A_4=\begin{pmatrix} 0 & 1 & 0 & 0 & 0 & 0 \\
 0 & 0 & 0 & 0 & 0 & 0 \\
 0 & 0 & 1 & 0 & 0 & 0 \\
 0 & 0 & 0 & -1 & 1 & 0 \\
 0 & 0 & 0 & 0 & -1 & 1 \\
 0 & 0 & 0 & 0 & 0 & -1 \end{pmatrix}.\vspace{2mm}
 $$ 
 Then the operator $A_4$ \vspace{2mm}is self-adjoint in $\mathcal K_4$ and with $\Gamma_4=\begin{pmatrix} \frac{1}{2} & -1 \\
 1 & 0 \\
 1 & 0 \\
 0 & -\frac{1}{2} \\
 0 & -1 \\
 0 & 1 \end{pmatrix}$ the representation   $Q_4(z)=\Gamma_4^+(A_4-z)^{-1}\Gamma_4$ holds.
\\
Any Jordan chain of length ${  2}$ is of the form 
$$
\vec x_0:= \begin{pmatrix} 1 \\ 0 \\0\\0\\0\\0 \end{pmatrix},\quad 
\vec x_1:= \begin{pmatrix} a \\ 1 \\0\\0\\0\\0 \end{pmatrix}  
$$
with ${ a} \in\mathbb C$. Any  function $\vec\eta(z)$ defined as in (3.4) is of the form $Q_4(z)^{-1}\Gamma_4^+\big({ \vec x_0+z \vec x_1 }\big)$, which equals \vspace{1mm}
$$
Q_4(z)^{-1}\Gamma_4^+\big({  \vec x_0+z \vec x_1 }\big)=\begin{pmatrix} -\dfrac{(z-1) \left({ 2 z^2+ (4-4a)z +1  - 2 a  
} \right)}{2 (z+2)} \\[3mm]\dfrac{(z+1)^3 \left({  (2 a +1) z+1-2a
 }\right)}{2 z (z+2)}\end{pmatrix}.
$$

Hence for no Jordan chain of length ${ 2}$ this function  is a pole cancellation function, as already condition \textbf{(A)} {cannot be } satisfied.

\end{example}
In this example the order of the zero $\alpha=0$ is larger than the order of the pole $\alpha=0$.   In a similar way examples can be constructed where $\vec\eta$ still is a pole cancellation function, but its order is reduced. In those examples  the order of the pole is still larger than the order of the zero.

\section{Conclusions}
The results in this paper - in particular - the  construction of the pole cancellation function  of higher order gives a complete answer to the longstanding problem of an analytic characterization of generalized poles including multiplicities. The concrete form of the pole cancellation functions appears to be much simpler than expected.  In the following theorem we summarize the situation.

\begin{theorem}
Let $Q\in\mathcal N_\kappa(\mathcal H)$ and $\alpha\in\mathbb R$ be given. Then the following statements are equivalent:
\begin{itemize}
 \item[{\bf I}] The point  $\alpha$ is a generalized pole of $Q$  and there exists a Jordan chain of the representing relation of length $\ell$. 
 \item[{\bf II}] There exists a pole cancellation function of $Q$ at $\alpha$ of order at least $\ell$.
\item[{\bf III}] There exists a strong pole cancellation function of $Q$ at $\alpha$ of order at least \nolinebreak $\ell$.
\item[{\bf IV}] There exist $S=S^*\in\mathcal L(\mathcal H)$  and an $\mathcal H$-valued polynomial $\vec p(z)$ of degree $<\ell$ such that $(Q(z)+S)^{-1}\vec p(z)$ is a strong pole cancellation function of $Q$ at $\alpha$ of order at least $\ell$.
\end{itemize}

\end{theorem}

The main result of this paper concerns the implications  {\bf I}$\Rightarrow${\bf IV} and {\bf II}$\Rightarrow${\bf I}. But we want to mention that also      {\bf II}$\Rightarrow${\bf III} is new, even for $\ell=1$.

{ 
Pole cancellation functions have already been used as a tool in spectral theory, as recent examples can be mentioned \cite{BLu, BLT}. We expect that the results from the current paper will enhance this, as now also characterizations for multiplicities are available in the most general case and the treatment can be simplified due to the concrete form of a pole cancellation function. 
}


\begin{thebibliography}{99}
 
\bibitem{BLu} { J. Behrndt and A. Luger}, {\it An analytic characterization of the eigenvalues of self-adjoint extensions}, J.\,Funct.\,Anal.\,242 (2007), no. 2, 607–640. 

\bibitem{BLT} { J.  Behrndt, A. Luger, and C. Trunk}, {\it On the negative squares of a class of self-adjoint extensions in Krein spaces}‚  Math.\,Nachr. 286 (2013), no. 2-3, 118–148. 


\bibitem{B}  {M. Borogovac},  {\it Multiplicities of nonsimple zeros of meromorphic matrix 
functions of the class $ N_{\kappa }^{n\times n}$}, Math.Nachr.\,\textbf{153}  (1991), 69--77. 

\bibitem{BL}  {M. Borogovac \and H. Langer},  {\it A characterization of generalized zeros of 
negative type of matrix functions of the class $N_{\kappa }^{n\times n}$}, Oper.~Theory Adv.~Appl.\,\textbf{28}  (1988), 17--26. 

\bibitem{DaL} { K. Daho \and H. Langer},
{\it Matrix functions of the class $N\sb \kappa$},
Math.Nachr.\ \textbf{120} (1985), 275--294.

\bibitem{DLS93} { A. Dijksma, H. Langer \and H.S.V.~de~Snoo}, {\it Eigenvalues and pole functions of Hamiltonian systems with eigenvalue depending boundary conditions},  Math.~Nachr.~\textbf{161}  (1993), 107--154.

\bibitem{DS} { A. Dijksma \and H.S.V. de Snoo}, {\it Symmetric and selfadjoint relations in Krein spaces}, Integr. Equ.~Oper.~Theory \textbf{24} (1987), 145--166.

\bibitem{GLR}{ I.C. Gohberg, P. Lancaster, and L. Rodman}, {\it  Matrix polynomials. Computer Science and Applied Mathematics} , Academic Press, New York-London, 1982. xiv+409 pp. 

\bibitem{GSi}  {  I.C. Gohberg \and E.I. Sigal}, {\it An operator generalization of the logarithmic residue theorem and the theorem of Rouche}, Math.~USSR-Sb.\,\textbf{13}  (1971), 603--652.

\bibitem{HL}  {S. Hassi \and A. Luger},  {\it Generalized zeros and poles of $\mathcal N_{\kappa }$-functions: On the underlying spectral structure}, Methods Funct.\,Anal.\,Topology, \textbf{12} (2006) 2, 131--150.

\bibitem{HSW} { S. Hassi, H.S.V. de Snoo \and H. Woracek},  {\it Some interpolation problems of Nevanlinna-Pick type}, Oper.~Theory Adv.~Appl.\,\textbf{106}  (1998), 201--216.

\bibitem{KL2} { M.G. Krein \and H. Langer},  {\it \"{U}ber einige Fortsetzungsprobleme, 
die eng mit der Theorie hermitescher Operatoren im Raume $\Pi_{\kappa }$ 
zusammenh\"angen, I. Einige Funktionenklassen und ihre Darstellungen}, Math.~Nachr.~\textbf{77} (1977), 187--236 . 

\bibitem{L2} { H. Langer},  {\it A characterization of generalized zeros of negative type of 
functions of the class $N_{\kappa }$}, Oper.~Theory Adv.~Appl.\,\textbf{17} (1986), 201--212.

\bibitem{LaLu} { H. Langer \and A. Luger},  {\it A class of $2\times 2$-matrix functions}, Glas.~Mat.~Ser.~III, \textbf{35}(55) (2000), 149--160.

\bibitem{Lu1} { A. Luger},  {\it A factorization of regular generalized Nevanlinna functions}, 
Integr.~Equ.~Oper.\ Theory \textbf{43}  (2002), 326--345.

\bibitem{Lu2} {A. Luger},  {\it A characterization of generalized poles of generalized 
Nevanlinna functions}, Math.~Nachr. \textbf{279} (2006), 891--910.


\end{thebibliography}
\end{document}